\documentclass[12pt]{amsart}
\usepackage{amsmath,amsthm,amsfonts,amscd,amssymb,eucal,latexsym,mathrsfs, appendix, yhmath, setspace}
\usepackage[numbers,sort&compress]{natbib}
\usepackage[all,cmtip]{xy}
\usepackage{enumerate}
\newtheorem{theorem}{Theorem}[section]

\newtheorem{lemma}[theorem]{Lemma}
\newtheorem{proposition}[theorem]{Proposition}

\theoremstyle{definition}
\newtheorem{definition}[theorem]{Definition}
\newtheorem{remark}[theorem]{Remark}

\newtheorem{example}[theorem]{Example}

\allowdisplaybreaks
\begin{document}

\title{The dynamic behavior of conjugate multipliers on some reflexive Banach spaces of analytic functions}

\author{Zhen Rong}

\address{\hskip-\parindent
Z.R., College of statistics and mathematics, Inner Mongolia University of Finance and Economics,
Hohhot 010000, China.}
\email{rongzhenboshi@sina.com}

\date{June 1, 2023}


\begin{abstract}
Extending previous results of Godefroy and Shapiro we characterize the hypercyclic, mixing and chaotic conjugate multipliers on some reflexive spaces of analytic functions.
\end{abstract}

\maketitle

%%%%%%%%%%%%%%%%%%%%%%%%%%%%%%%%%%%%%%%%%%%%%%%%%%%%%%%%%%%%%%%%%%%%%%%%%%%%%%%%%%%%%%%%%%%%%%%%%%%%%%%%%%%%%%%%%%%%%%%%%
\section{Introduction} \label{S-introduction}
Throughout this article, let $\mathbb{N}$ denote the set of nonnegative integers. Let $\mathbb{K}$ denote the real number field $\mathbb{R}$ or the complex number field $\mathbb{C}$. Let $\mathbb{Q}$ denote the rational number field. If $z\in\mathbb{C}$ and $r>0$ are fixed then define $B(z,r)=\{\lambda\in\mathbb{C}:|\lambda-z|<r\}$. Let $\mathbb{T}=\{z\in\mathbb{C}:|z|=1\}$ and $\mathbb{D}=\{z\in\mathbb{C}:|z|<1\}$.

A continuous linear operator $T$ on a Banach space $X$ is called {\it hypercyclic} if there is an element $x$ in $X$ whose orbit $\{T^{n}x:n\in\mathbb{N}\}$ under $T$ is dense in $X$; {\it topologically transitive} if for any pair $U,V$ of nonempty open subsets of $X$, there exists some nonnegative integer $n$ such that $T^{n}(U)\cap V\neq\emptyset$; {\it mixing} if for any pair $U,V$ of nonempty open subsets of $X$, there exists some nonnegative integer $N$ such that $T^{n}(U)\cap V\neq\emptyset$ for all $n\geqslant N$; and {\it chaotic} if $T$ is topologically transitive and $T$ has a dense set of periodic points.

The historical interest in hypercyclicity is related to the invariant subset problem. The invariant subset problem, which is open to this day, asks whether every continuous linear operator on any infinite dimensional separable Hilbert space possesses an invariant closed subset other than the trivial ones given by $\{0\}$ and the whole space. Counterexamples do exist for continuous linear operators on non-reflexive spaces like $l^{1}$. After a simple observation, a continuous linear operator $T$ on a Banach space $X$ has no nontrival invariant closed subsets if and only if every nonzero vector $x$ is hypercyclic (i.e. the orbit $\{T^{n}x:n\in\mathbb{N}\}$ under $T$ is dense in $X$).

The best known examples of hypercyclic operators are due to Birkhoff \cite{Birkhoff}, MacLane \cite{MacLane} and Rolewicz \cite{Rolewicz}. Each of these papers had a profound influence on the literature on hypercyclicity. Birkhoff's result on the hypercyclicity of the translation operator $T_{a}(f)(z)=f(z+a),a\neq0,$ on the space $H(\mathbb{C})$ of entire functions has led to an extensive study of hypercyclic composition operators (see \cite[pages 110-118]{Grosse-Erdmann-Peris}), while MacLane's result on the hypercyclicity of the differentiation operator $Df=f^{\prime}$ on $H(\mathbb{C})$ initiated the study of hypercyclic differential operators (see \cite[pages 104-110]{Grosse-Erdmann-Peris}).

Recently Godefroy and Shapiro \cite{Godefroy-Shapiro} have studied the dynamic properties of conjugate multipliers on some Hilbert spaces of analytic functions. Godefroy and Shapiro \cite{Godefroy-Shapiro} characterized hypercyclic, mixing and chaotic conjugate multipliers on some Hilbert spaces of analytic functions. It is therefore very natural to try to characterize hypercyclic, mixing and chaotic conjugate multipliers on arbitrary reflexive Banach spaces of analytic functions. In this paper we will characterize the hypercyclic, mixing and chaotic conjugate multipliers on some reflexive spaces of analytic functions, generalizing \cite[Theorem 4.5, Theorem 4.9, Theorem 6.2]{Godefroy-Shapiro}.

\begin{theorem}
Let $\Omega\subseteq\mathbb{C}$ be a nonempty open connected subset. Let $X\neq\{0\}$ be a reflexive Banach space of analytic functions on $\Omega$ such that each point evaluation $k_{\lambda}:X\rightarrow\mathbb{C}(\lambda\in\Omega)$ is continuous on $X$, where $k_{\lambda}(f)=f(\lambda)(f\in X)$. Suppose further that every bounded analytic function $\psi$ on $\Omega$ defines a multiplication operator $M_{\psi}:X\rightarrow X$ with $\|M_{\psi}\|\leqslant\sup\limits_{z\in\Omega}|\psi(z)|$, where $M_{\psi}(f)=\psi f(f\in X)$. Let $\varphi$ be a nonconstant bounded analytic function on $\Omega$ and $M_{\varphi}^{\ast}$ the conjugate of $M_{\varphi}$. Then the following assertions are equivalent:

(1) $M_{\varphi}^{\ast}$ is hypercyclic;

(2) $M_{\varphi}^{\ast}$ is mixing;

(3) $M_{\varphi}^{\ast}$ is chaotic;

(4) $\varphi(\Omega)\cap\mathbb{T}\neq\emptyset$.
\end{theorem}

Godefroy and Shapiro \cite[Theorem 4.5, Theorem 4.9, Theorem 6.2]{Godefroy-Shapiro} proved Theorem 1.1 in the case of Hilbert spaces of analytic functions, thus Theorem 1.1 generalizes \cite[Theorem 4.5, Theorem 4.9, Theorem 6.2]{Godefroy-Shapiro}.

This paper is organized as follows. In Section~\ref{S-analytic} we characterize the hypercyclic, mixing and chaotic conjugate multipliers on some reflexive spaces of analytic functions, generalizing \cite[Theorem 4.5, Theorem 4.9, Theorem 6.2]{Godefroy-Shapiro}. Furthermore, we exhibit several hypercyclic, mixing and chaotic conjugate multipliers on $H^{p}$ spaces for $p>1$. These examples show that our generalizations are more effective.

\noindent{\it Acknowledgments.}
Z.~R. was supported by National Natural Science Foundation of China (Grant No.12261063).

\section{The dynamic behavior of conjugate multipliers on some reflexive Banach spaces of analytic functions} \label{S-analytic}
In this section we characterize hypercyclic, mixing and chaotic conjugate multipliers on some reflexive Banach spaces of analytic functions, generalizing \cite[Theorem 4.5, Theorem 4.9, Theorem 6.2]{Godefroy-Shapiro}.

Recall the notion of annihilator introduced in \cite[page 163]{Taylor-Lay}.

\begin{definition}
Let $X$ be a normed linear space. If $A\subseteq X$, the annihilator $A^{\bot}$ of $A$ is the set
$$A^{\bot}=\{x^{\prime}\in X^{\ast}:x^{\prime}(x)=0\text{ for all }x\in A\},$$
where $X^{\ast}$ is the set of continuous linear functionals on $X$.

If $F\subseteq X^{\ast}$, the annihilator $F^{\bot}$ of $F$ is the set
$$F^{\bot}=\{x\in X:x^{\prime}(x)=0\text{ for all }x^{\prime}\in F\}.$$

\end{definition}

The following technical results will help us characterize hypercyclic, mixing and chaotic conjugate multipliers on some reflexive Banach spaces of analytic functions.

The following proposition is well known (see \cite[page 164]{Taylor-Lay}).

\begin{proposition}
A normed linear space $X$ is a reflexive Banach space if and only if every norm-closed linear subspace in $X^{\ast}$ is $\sigma(X^{\ast},X)$-closed, where $\sigma(X^{\ast},X)$ is the weak$^{\ast}$ topology on $X^{\ast}$.
\end{proposition}

We need the following proposition (see \cite[pages 163-164]{Taylor-Lay}).

\begin{proposition}
Let $X$ be a normed linear space. If $F$ is a nonempty subset of $X^{\ast}$, then $F^{\bot\bot}$ is the $\sigma(X^{\ast},X)$-closed linear subspace generated by $F$, where $F^{\bot\bot}=(F^{\bot})^{\bot}$.
\end{proposition}

The following proposition is well known in complex analysis (see \cite[page 78]{Conway}).

\begin{proposition}
Let $\Omega\subseteq\mathbb{C}$ be a nonempty open connected subset. Let $f:\Omega\rightarrow\mathbb{C}$ be an analytic function on $\Omega$. Then the following are equivalent statements:
\begin{enumerate}
  \item $f\equiv0$;
  \item $\{z\in\Omega:f(z)=0\}$ has a limit point in $\Omega$.
\end{enumerate}
\end{proposition}

We need the following Godefroy-Shapiro criterion (see \cite[pages 69-70]{Grosse-Erdmann-Peris}).

\begin{proposition}
Let $T$ be a continuous linear operator on a separable Banach space $X$. Suppose that the subspaces
$$X_{0}=span\{x\in X:Tx=\lambda x\text{ for some }\lambda\in\mathbb{K}\text{ with }|\lambda|<1\},$$
$$Y_{0}=span\{x\in X:Tx=\lambda x\text{ for some }\lambda\in\mathbb{K}\text{ with }|\lambda|>1\}$$
are dense in $X$. Then $T$ is mixing, and in particular hypercyclic.

If, moreover, $X$ is a complex space and also the subspace
$$Z_{0}=span\{x\in X:Tx=e^{\alpha\pi i}x\text{ for some }\alpha\in\mathbb{Q}\}$$
is dense in $X$, then $T$ is chaotic.
\end{proposition}

The following is the major technique we need.

\begin{lemma}
Let $\Omega\subseteq\mathbb{C}$ be a nonempty open connected subset. Let $X\neq\{0\}$ be a reflexive Banach space of analytic functions on $\Omega$ such that each point evaluation $k_{\lambda}:X\rightarrow\mathbb{C}(\lambda\in\Omega)$ is continuous on $X$, where $k_{\lambda}(f)=f(\lambda)(f\in X)$. Let $\Lambda\subseteq\Omega$ be a set with a limit point in $\Omega$. Then the set $span\{k_{\lambda}:\lambda\in\Lambda\}$ is dense in $X^{\ast}$.
\end{lemma}

\begin{proof}
First we will show that $(span\{k_{\lambda}:\lambda\in\Lambda\})^{\bot}=\{0\}$. Let $f\in(span\{k_{\lambda}:\lambda\in\Lambda\})^{\bot}$. We will show that $f\equiv0$. Since $f\in(span\{k_{\lambda}:\lambda\in\Lambda\})^{\bot}$, we have $k_{\lambda}(f)=0$ for all $\lambda\in\Lambda$. Notice that $k_{\lambda}(f)=f(\lambda)$. Then $f(\lambda)=0$ for all $\lambda\in\Lambda$. This implies that $\Lambda\subseteq\{z\in\Omega:f(z)=0\}$. Since $\Lambda$ has a limit point in $\Omega$, $\{z\in\Omega:f(z)=0\}$ has a limit point in $\Omega$. By Proposition 2.4 we have $f\equiv0$.

Next we will show that $\overline{span\{k_{\lambda}:\lambda\in\Lambda\}}=X^{\ast}$. Since $$(\overline{span\{k_{\lambda}:\lambda\in\Lambda\}})^{\bot}=(span\{k_{\lambda}:\lambda\in\Lambda\})^{\bot}$$
and $(span\{k_{\lambda}:\lambda\in\Lambda\})^{\bot}=\{0\}$, $(\overline{span\{k_{\lambda}:\lambda\in\Lambda\}})^{\bot}=\{0\}$. Hence $(\overline{span\{k_{\lambda}:\lambda\in\Lambda\}})^{\bot\bot}=\{0\}^{\bot}=X^{\ast}$. Since $X$ is reflexive and $\overline{span\{k_{\lambda}:\lambda\in\Lambda\}}$ is norm-closed, by Proposition 2.2 we have $\overline{span\{k_{\lambda}:\lambda\in\Lambda\}}$ is $\sigma(X^{\ast},X)$-closed. Finally by Proposition 2.3 we have $$(\overline{span\{k_{\lambda}:\lambda\in\Lambda\}})^{\bot\bot}=\overline{span\{k_{\lambda}:\lambda\in\Lambda\}}.$$
Therefore $\overline{span\{k_{\lambda}:\lambda\in\Lambda\}}=X^{\ast}$.
\end{proof}

\begin{flushright}
  $\Box$
\end{flushright}

\begin{remark}
Let $\Omega\subseteq\mathbb{C}$ be a nonempty open connected subset. Let $X\neq\{0\}$ be a reflexive Banach space of analytic functions on $\Omega$ such that each point evaluation $k_{\lambda}:X\rightarrow\mathbb{C}(\lambda\in\Omega)$ is continuous on $X$. Then by Lemma 2.6 we have $X^{\ast}$ is separable.
\end{remark}

Next we prove Theorem 1.1.

$\mathbf{Proof~of~Theorem~1.1.}$

(4)$\Rightarrow$(2) First we will show that $M_{\varphi}^{\ast}(k_{\lambda})=\varphi(\lambda)k_{\lambda}$ for all $\lambda\in\Omega$. Let $\lambda\in\Omega$ and $f\in X$. Notice that

\begin{align*}
M_{\varphi}^{\ast}(k_{\lambda})(f)=&k_{\lambda}(M_{\varphi}(f))\\
=&k_{\lambda}(\varphi f)\\
=&(\varphi f)(\lambda)\\
=&\varphi(\lambda)f(\lambda)\\
=&(\varphi(\lambda)k_{\lambda})(f).
\end{align*}
Therefore $M_{\varphi}^{\ast}(k_{\lambda})=\varphi(\lambda)k_{\lambda}$ for all $\lambda\in\Omega$.

Next we will show that $\{\lambda\in\Omega:|\varphi(\lambda)|<1\}$ is nonempty and has a limit point in $\Omega$. Since $\varphi(\Omega)\cap\mathbb{T}\neq\emptyset$, we may choose $z_{0}\in\Omega$ with $|\varphi(z_{0})|=1$. Since nonconstant analytic functions are open mappings (see \cite[page 99]{Conway}), $\varphi(\Omega)$ is open. Hence there exists $\delta>0$ such that $B(\varphi(z_{0}),\delta)\subseteq\varphi(\Omega)$. Since $|\varphi(z_{0})|=1$, we may choose $\lambda_{0}\in B(\varphi(z_{0}),\delta)$ with $|\lambda_{0}|<1$. Since $B(\varphi(z_{0}),\delta)\subseteq\varphi(\Omega)$, we have $\lambda_{0}\in\varphi(\Omega)$. Therefore there exists $w_{0}\in\Omega$ such that $\lambda_{0}=\varphi(w_{0})$. This implies that $w_{0}\in\{\lambda\in\Omega:|\varphi(\lambda)|<1\}$ and $\{\lambda\in\Omega:|\varphi(\lambda)|<1\}$ is nonempty. Since $\varphi$ is continuous, $\{\lambda\in\Omega:|\varphi(\lambda)|<1\}$ is a nonempty open subset. Hence $\{\lambda\in\Omega:|\varphi(\lambda)|<1\}$ has a limit point in $\Omega$.

Similarly we can show that $\{\lambda\in\Omega:|\varphi(\lambda)|>1\}$ is nonempty and has a limit point in $\Omega$.

Finally we will show that $M_{\varphi}^{\ast}$ is mixing. Let
$$A=span\{x^{\ast}\in X^{\ast}:M_{\varphi}^{\ast}x^{\ast}=\lambda x^{\ast}\text{ for some }\lambda\in\mathbb{C}\text{ with }|\lambda|<1\},$$
$$B=span\{x^{\ast}\in X^{\ast}:M_{\varphi}^{\ast}x^{\ast}=\lambda x^{\ast}\text{ for some }\lambda\in\mathbb{C}\text{ with }|\lambda|>1\}.$$
Since $M_{\varphi}^{\ast}(k_{\lambda})=\varphi(\lambda)k_{\lambda}$ for all $\lambda\in\Omega$, $span\{k_{\lambda}:\lambda\in\Omega\text{ and }|\varphi(\lambda)|<1\}\subseteq A$ and $span\{k_{\lambda}:\lambda\in\Omega\text{ and }|\varphi(\lambda)|>1\}\subseteq B$. Since $\{\lambda\in\Omega:|\varphi(\lambda)|<1\}$ and $\{\lambda\in\Omega:|\varphi(\lambda)|>1\}$ both have a limit point in $\Omega$, by Lemma 2.6 we have $span\{k_{\lambda}:\lambda\in\Omega\text{ and }|\varphi(\lambda)|<1\}$ and $span\{k_{\lambda}:\lambda\in\Omega\text{ and }|\varphi(\lambda)|>1\}$ are both dense in $X^{\ast}$. Hence $A$ and $B$ are both dense in $X^{\ast}$. By Proposition 2.5, we have $M_{\varphi}^{\ast}$ is mixing.

(2)$\Rightarrow$(1) Assume that $M_{\varphi}^{\ast}$ is mixing. Then $M_{\varphi}^{\ast}$ is topologically transitive. Since a continuous linear operator on a separable Banach space is topologically transitive if and only if it is hypercyclic (see \cite[page 10]{Grosse-Erdmann-Peris}), by Remark 2.7 we have $M_{\varphi}^{\ast}$ is hypercyclic.

(1)$\Rightarrow$(4) Let us suppose that $\varphi(\Omega)$ does not intersect the unit circle. Since $\varphi(\Omega)$ is open and connected, it must lie entirely inside or entirely outside $\mathbb{D}$. If $\varphi(\Omega)\subseteq\mathbb{D}$ then
$$\|M_{\varphi}^{\ast}\|=\|M_{\varphi}\|\leqslant\sup\limits_{z\in\Omega}|\varphi(z)|\leqslant1,$$
and hence $M_{\varphi}^{\ast}$ cannot be hypercyclic. If $\varphi(\Omega)\subseteq\mathbb{C}\setminus\overline{\mathbb{D}}$ then $\psi=\frac{1}{\varphi}$ is a bounded analytic function on $\Omega$ with $\psi(\Omega)\subseteq\mathbb{D}$, which implies that $M_{\psi}^{\ast}$ cannot be hypercyclic. But $M_{\varphi}$ is the inverse of $M_{\psi}$ and therefore $M_{\varphi}^{\ast}$ is the inverse of $M_{\psi}^{\ast}$. Hence $M_{\varphi}^{\ast}$ cannot be hypercyclic.

(4)$\Rightarrow$(3) First we will show that $\{\lambda\in\Omega:\varphi(\lambda)\text{ is a root of unity}\}$ has a limit point in $\Omega$. Since $\varphi(\Omega)\cap\mathbb{T}\neq\emptyset$, we may choose $z_{0}\in\Omega$ with $|\varphi(z_{0})|=1$. Since $\Omega$ is open, there is an $r>0$ such that $\overline{B(z_{0},r)}\subseteq\Omega$. Since nonconstant analytic functions are open mappings (see \cite[page 99]{Conway}), $\varphi(B(z_{0},r))$ is open. Since $\varphi(z_{0})\in\varphi(B(z_{0},r))$, there is an $\delta>0$ such that $B(\varphi(z_{0}),\delta)\subseteq\varphi(B(z_{0},r))$. Since $|\varphi(z_{0})|=1$, $B(\varphi(z_{0}),\delta)$ contains infinitely many roots of unity. This implies that infinitely many preimages of roots of unity lie in the compact subset $\overline{B(z_{0},r)}$ of $\Omega$. Therefore $\{\lambda\in\Omega:\varphi(\lambda)\text{ is a root of unity}\}$ has a limit point in $\Omega$.

Next we will show that $M_{\varphi}^{\ast}$ is chaotic. Let
$$C=span\{x^{\ast}\in X^{\ast}:M_{\varphi}^{\ast}x^{\ast}=e^{\alpha\pi i}x^{\ast}\text{ for some }\alpha\in\mathbb{Q}\}.$$
Since $M_{\varphi}^{\ast}(k_{\lambda})=\varphi(\lambda)k_{\lambda}$ for all $\lambda\in\Omega$, $span\{k_{\lambda}:\lambda\in\Omega\text{ and } \varphi(\lambda)\text{ is a root of unity}\}\subseteq C$. Since $\{\lambda\in\Omega:\varphi(\lambda)\text{ is a root of unity}\}$ has a limit point in $\Omega$, by Lemma 2.6 we have $span\{k_{\lambda}:\lambda\in\Omega\text{ and } \varphi(\lambda)\text{ is a root of unity}\}$ is dense in $X^{\ast}$. Hence $C$ is dense in $X^{\ast}$. By (4)$\Rightarrow$(2), we have $M_{\varphi}^{\ast}$ is mixing. By Proposition 2.5, we have $M_{\varphi}^{\ast}$ is chaotic.

(3)$\Rightarrow$(1) Assume that $M_{\varphi}^{\ast}$ is chaotic. Then $M_{\varphi}^{\ast}$ is topologically transitive. Since a continuous linear operator on a separable Banach space is topologically transitive if and only if it is hypercyclic (see \cite[page 10]{Grosse-Erdmann-Peris}), by Remark 2.7 we have $M_{\varphi}^{\ast}$ is hypercyclic.

\begin{flushright}
  $\Box$
\end{flushright}

Godefroy and Shapiro \cite[Theorem 4.5, Theorem 4.9, Theorem 6.2]{Godefroy-Shapiro} proved Theorem 1.1 in the case of Hilbert spaces of analytic functions, thus Theorem 1.1 generalizes \cite[Theorem 4.5, Theorem 4.9, Theorem 6.2]{Godefroy-Shapiro}.

\begin{example}
For $1\leqslant p<+\infty$, let $H^{p}$ denote the space of all analytic functions on $\mathbb{D}$ for which $\sup\limits_{0\leqslant r<1}(\frac{1}{2\pi}\int_{0}^{2\pi}|f(re^{i\theta})|^{p}d\theta)^{\frac{1}{p}}<+\infty$. For any $f\in H^{p}$, let
$$\|f\|_{p}=\sup\limits_{0\leqslant r<1}(\frac{1}{2\pi}\int_{0}^{2\pi}|f(re^{i\theta})|^{p}d\theta)^{\frac{1}{p}}.$$
Then $(H^{p},\|\cdot\|_{p})$ is a Banach space.

In this example we will characterize hypercyclic, mixing and chaotic conjugate multipliers on $H^{p}$ for $1<p<+\infty$, generalizing \cite[page 253]{Godefroy-Shapiro}.

First we will show that each point evaluation $k_{\lambda}:H^{p}\rightarrow\mathbb{C}(\lambda\in\mathbb{D})$ is continuous on $H^{p}$ for $1\leqslant p<+\infty$, where $k_{\lambda}(f)=f(\lambda)(f\in H^{p})$. Let $\lambda\in\mathbb{D}$, $\{f_{n}\}_{n=1}^{\infty}$ be a sequence in $H^{p}$, $f\in H^{p}$ and $\lim\limits_{n\rightarrow\infty}\|f_{n}-f\|_{p}=0$. We will show that $\lim\limits_{n\rightarrow\infty}f_{n}(\lambda)=f(\lambda)$. Since $\lambda\in\mathbb{D}$, we may choose $r,R\in(0,1)$ with $|\lambda|<r<R<1$. Notice that
\begin{align*}
\|f_{n}-f\|_{p}\geqslant&(\frac{1}{2\pi}\int_{0}^{2\pi}|f_{n}(Re^{i\theta})-f(Re^{i\theta})|^{p}d\theta)^{\frac{1}{p}}\\
\geqslant&\frac{1}{(2\pi)^{\frac{1}{p}}}(\int_{0}^{2\pi}|f_{n}(Re^{i\theta})-f(Re^{i\theta})|^{p}d\theta)^{\frac{1}{p}}\\
\geqslant&\frac{1}{(2\pi)^{\frac{1}{p}}}\frac{1}{(2\pi)^{\frac{1}{q}}}\int_{0}^{2\pi}|f_{n}(Re^{i\theta})-f(Re^{i\theta})|d\theta(\text{ where }\frac{1}{p}+\frac{1}{q}=1)\\
=&\frac{1}{2\pi}\int_{0}^{2\pi}|f_{n}(Re^{i\theta})-f(Re^{i\theta})|d\theta.
\end{align*}
By Cauchy's integral formula, we have
\begin{align*}
(f_{n}-f)(\lambda)=&\frac{1}{2\pi i}\int_{\gamma}\frac{(f_{n}-f)(\omega)}{\omega-\lambda}d\omega(\text{ where }\gamma(t)=Re^{i\theta},0\leqslant\theta\leqslant2\pi)\\
=&\frac{1}{2\pi i}\int_{0}^{2\pi}\frac{f_{n}(Re^{i\theta})-f(Re^{i\theta})}{Re^{i\theta}-\lambda}d(Re^{i\theta})\\
=&\frac{R}{2\pi}\int_{0}^{2\pi}\frac{f_{n}(Re^{i\theta})-f(Re^{i\theta})}{Re^{i\theta}-\lambda}e^{i\theta}d\theta.
\end{align*}
Therefore
\begin{align*}
&\frac{1}{2\pi}\int_{0}^{2\pi}|f_{n}(Re^{i\theta})-f(Re^{i\theta})|d\theta\\
&=\frac{1}{2\pi}\int_{0}^{2\pi}|\frac{f_{n}(Re^{i\theta})-f(Re^{i\theta})}{Re^{i\theta}-\lambda}e^{i\theta}|\cdot |Re^{i\theta}-\lambda|d\theta\\
&\geqslant\frac{R-r}{2\pi}\int_{0}^{2\pi}|\frac{f_{n}(Re^{i\theta})-f(Re^{i\theta})}{Re^{i\theta}-\lambda}e^{i\theta}|d\theta\\
&\geqslant\frac{R-r}{2\pi}|\int_{0}^{2\pi}\frac{f_{n}(Re^{i\theta})-f(Re^{i\theta})}{Re^{i\theta}-\lambda}e^{i\theta}d\theta|\\
&=\frac{R-r}{2\pi}|\frac{2\pi}{R}(f_{n}(\lambda)-f(\lambda))|\\
&=\frac{R-r}{R}|f_{n}(\lambda)-f(\lambda)|.
\end{align*}
Hence $\|f_{n}-f\|_{p}\geqslant\frac{R-r}{R}|f_{n}(\lambda)-f(\lambda)|$. Since $\lim\limits_{n\rightarrow\infty}\|f_{n}-f\|_{p}=0$, $\lim\limits_{n\rightarrow\infty}f_{n}(\lambda)=f(\lambda)$.

Next we will show that every bounded analytic function $\psi$ on $\mathbb{D}$ defines a multiplication operator $M_{\psi}:H^{p}\rightarrow H^{p}$ with $\|M_{\psi}\|\leqslant\sup\limits_{z\in\mathbb{D}}|\psi(z)|$ for $1\leqslant p<+\infty$. Let $\psi$ be a bounded analytic function on $\mathbb{D}$. Then $\psi$ defines a multiplication operator $M_{\psi}$ on $H^{p}$. Furthermore, for any $f\in H^{p}$ we have
\begin{align*}
\|M_{\psi}(f)\|_{p}&=\sup\limits_{0\leqslant r<1}(\frac{1}{2\pi}\int_{0}^{2\pi}|\psi(re^{i\theta})f(re^{i\theta})|^{p}d\theta)^{\frac{1}{p}}\\
&\leqslant\sup\limits_{z\in\mathbb{D}}|\psi(z)|\cdot\sup\limits_{0\leqslant r<1}(\frac{1}{2\pi}\int_{0}^{2\pi}|f(re^{i\theta})|^{p}d\theta)^{\frac{1}{p}}\\
&=\sup\limits_{z\in\mathbb{D}}|\psi(z)|\cdot\|f\|_{p}.
\end{align*}
Therefore $\|M_{\psi}\|\leqslant\sup\limits_{z\in\mathbb{D}}|\psi(z)|$.

Finally we will show that $(H^{p},\|\cdot\|_{p})$ is reflexive for $1<p<+\infty$. If $f\in H^{p}$, then $\lim\limits_{r\rightarrow1^{-}}f(re^{i\theta})$ exists for almost all values of $\theta$ (see \cite[page 17]{Duren}), thus defining a function which we denote by $f(e^{i\theta})$. If $1<p<+\infty$, each $\phi\in(H^{p})^{\ast}$ is representable in the following form
$$\phi(f)=\frac{1}{2\pi}\int_{0}^{2\pi}f(e^{i\theta})g(e^{-i\theta})d\theta(f\in H^{p})$$
by a unique function $g\in H^{q}$, where $\frac{1}{p}+\frac{1}{q}=1$. Furthermore, the linear operator $T:(H^{p})^{\ast}\rightarrow H^{q}$ defined by $T(\phi)=g(\phi\in(H^{p})^{\ast})$ is a topological isomorphism (see \cite[pages 112-113]{Duren}). Let $x^{\prime\prime}\in(H^{p})^{\ast\ast}$. We will show that there exists a $g\in H^{p}$ such that $x^{\prime\prime}(x^{\prime})=x^{\prime}(g)$ for all $x^{\prime}\in(H^{p})^{\ast}$. Let $y^{\prime\prime}=x^{\prime\prime}\circ T^{-1}$. Then $y^{\prime\prime}\in(H^{q})^{\ast}$ and $y^{\prime\prime}(Tx^{\prime})=x^{\prime\prime}(x^{\prime})$ for $x^{\prime}\in(H^{p})^{\ast}$. Hence there exists a unique function $g\in H^{p}$ such that
$$y^{\prime\prime}(f)=\frac{1}{2\pi}\int_{0}^{2\pi}f(e^{i\theta})g(e^{-i\theta})d\theta(f\in H^{q}).$$
Therefore, if $x^{\prime}\in(H^{p})^{\ast}$,
\begin{align*}
x^{\prime\prime}(x^{\prime})=&y^{\prime\prime}(Tx^{\prime})\\
=&\frac{1}{2\pi}\int_{0}^{2\pi}(Tx^{\prime})(e^{i\theta})g(e^{-i\theta})d\theta\\
=&-\frac{1}{2\pi}\int_{0}^{-2\pi}g(e^{i\theta})(Tx^{\prime})(e^{-i\theta})d\theta\\
=&\frac{1}{2\pi}\int_{-2\pi}^{0}g(e^{i\theta})(Tx^{\prime})(e^{-i\theta})d\theta\\
=&\frac{1}{2\pi}\int_{0}^{2\pi}g(e^{i\theta})(Tx^{\prime})(e^{-i\theta})d\theta\\
=&x^{\prime}(g).
\end{align*}
This implies that $x^{\prime\prime}(x^{\prime})=x^{\prime}(g)$ for all $x^{\prime}\in(H^{p})^{\ast}$ and $(H^{p},\|\cdot\|_{p})$ is reflexive.

By Theorem 1.1, for any nonconstant bounded analytic function $\varphi$ on $\mathbb{D}$ and $1<p<+\infty$, $M_{\varphi}^{\ast}:(H^{p})^{\ast}\rightarrow(H^{p})^{\ast}$ is hypercyclic if and only if $M_{\varphi}^{\ast}:(H^{p})^{\ast}\rightarrow(H^{p})^{\ast}$ is mixing if and only if $M_{\varphi}^{\ast}:(H^{p})^{\ast}\rightarrow(H^{p})^{\ast}$ is chaotic if and only if $\varphi(\mathbb{D})\cap\mathbb{T}\neq\emptyset$.
\end{example}

Godefroy and Shapiro \cite[page 253]{Godefroy-Shapiro} proved the above result in the case of $p=2$, thus the above example generalizes \cite[page 253]{Godefroy-Shapiro}.

%%%%%%%%%%%%%%%%%%%%%%%%%%%%%%%%%%%%%%%%%%%%%%%%%%%%%%%%%%%%%%%%%%%%%%%%%%%%%%%%%%%%%%%%%%%%%%%%%%%%%%%%%%%%%%%%%%%%%%%%%%%


\begin{thebibliography}{21}
\Small

\bibitem{Beauzamy}
Beauzamy, B.: Introduction to Operator Theory and Invariant Subspaces, North-Holland Publishing Co., Amsterdam, 1988

\bibitem{Bernal-Gonzalez}
Bernal-Gonz\'{a}lez, L.: Derivative and antiderivative operators and the size of complex domains. \emph{ Ann. Polon. Math.}, \textbf{59}, 267--274
(1994)

\bibitem{Bes-Peris}
B\`{e}s, J., Peris, A.: Hereditarily hypercyclic operators. \emph{J. Funct. Anal.}, \textbf{167}, 94--112
(1999)

\bibitem{Birkhoff}
Birkhoff, G. D.: D\'{e}monstration d'un th\'{e}or\`{e}me \'{e}l\'{e}mentaire sur les fonctions enti\`{e}res. \emph{C. R. Acad. Sci. Paris.}, \textbf{189}, 473--475
(1929)

\bibitem{Chan}
Chan, K. C., Sanders, R.: A weakly hypercyclic operator that is not norm hypercyclic. \emph{J. Operator Theory.}, \textbf{52}, 39--59
(2004)

\bibitem{Conway}
Conway, J. B.: Functions of One Complex Variable, Graduate Texts in Mathematics, 11, Springer-Verlag, New York, 1978

\bibitem{Costakis-Sambarino}
Costakis, G., Sambarino, M.: Topologically mixing hypercyclic operators. \emph{Proc. Amer. Math. Soc.}, \textbf{132}, 385--389
(2004)

\bibitem{Duren}
Duren, P. L.: Theory of $H^{p}$ Spaces, Academic Press, New York, 1970

\bibitem{Gethner-Shapiro}
Gethner, R. M., Shapiro, J. H.: Universal vectors for operators on spaces of holomorphic functions. \emph{Proc. Amer. Math. Soc.}, \textbf{100}, 281--288
(1987)

\bibitem{Godefroy-Shapiro}
Godefroy, G., Shapiro, J. H.: Operators with dense, invariant, cyclic vector manifolds. \emph{J. Funct. Anal.}, \textbf{98}, 229--269
(1991)

\bibitem{Grosse-Erdmann}
Grosse-Erdmann, K.-G.: Hypercyclic and chaotic weighted shifts. \emph{Studia Math.}, \textbf{139}, 47--68
(2000)

\bibitem{Grosse-Erdmann-Peris}
Grosse-Erdmann, K.-G., Peris, A.: Linear Chaos, Universitext, 223, Springer-Verlag, London, 2011

\bibitem{Gulisashvili-MacCluer}
Gulisashvili, A., MacCluer, C. R.: Linear chaos in the unforced quantum harmonic oscillator. \emph{J. Dynam. Systems Measurement Control.}, \textbf{118}, 337--338
(1996)

\bibitem{Kitai}
Kitai, C.: Invariant closed sets for linear operators, Thesis, University of Toronto, Toronto, 1982

\bibitem{MacLane}
MacLane, G. R.: Sequences of derivatives and normal families. \emph{J. Anal. Math.}, \textbf{2}, 72--87
(1952/53)

\bibitem{Martinez-Peris}
Mart\'{\i}nez-Gim\'{e}nez, F., Peris, A.: Chaos for backward shift operators. \emph{Internat. J. Bifur. Chaos Appl. Sci. Engrg.}, \textbf{12}, 1703--1715
(2002)

\bibitem{Mathew}
Mathew, V.: A note on hypercyclic operators on the space of entire sequences. \emph{Indian J. Pure Appl. Math.}, \textbf{25}, 1181--1184
(1994)

\bibitem{Rolewicz}
Rolewicz, S.: On orbits of elements. \emph{Studia Math.}, \textbf{32}, 17--22
(1969)

\bibitem{Salas91}
Salas, H.: A hypercyclic operator whose adjoint is also hypercyclic. \emph{Proc. Amer. Math. Soc.}, \textbf{112}, 765--770
(1991)

\bibitem{Salas95}
Salas, H.: Hypercyclic weighted shifts. \emph{Trans. Amer. Math. Soc.}, \textbf{347}, 993--1004
(1995)

\bibitem{Taylor-Lay}
Taylor, A. E., Lay, D. C.: Introduction to Functional Analysis, second edition, John Wiley-Sons, New York, 1980

\end{thebibliography}
\end{document}